\documentclass[preprint,12pt]{elsarticle}
\usepackage{mathtools}
\usepackage{booktabs}
\usepackage{graphicx} 
\usepackage{epstopdf} 
\usepackage{float} 
\usepackage{geometry} 
\usepackage{appendix} 
\usepackage{listings} 
\usepackage{setspace}
\usepackage{amssymb}
\usepackage{pdfpages}

\usepackage{xcolor} 
\usepackage{soul}   

\usepackage{tikz}
\usetikzlibrary{graphs, graphs.standard} 
\usepackage{tikz,mathpazo}
\usetikzlibrary{shapes.geometric, arrows}

\tikzset{%
	node/.style={circle, draw}, 
	edge/.style={->, thick}, 
	arrow/.style={line width=0.5pt}, 
}

\graphicspath{{qiquanfigs/}}
\counterwithin{figure}{section}

\newtheorem{lemma}{Lemma}[section]
\newtheorem{theorem}{Theorem}[section]
\newtheorem{corollary}{Corollary}[section]
\newtheorem{definition}{Definition}[section]
\newtheorem{proposition}{Proposition}[section]
\newtheorem{example}{Example}[section]
\newtheorem{proof}{Proof}[section]
\newtheorem{remark}{Remark}[section]

\def\bdefinition{\begin{definition}\sl{}\def\edefinition{\end{definition}}}

\usepackage{amssymb}

\begin{document}

\begin{frontmatter}

\title{On the critical group of the $k$-partite graph}

\author{Xinyu Dong}
\ead{dingzfcpp4141@126.com}

\affiliation{organization={College of Mathematics and Physics, Beijing University of Chemical Technology},
	addressline={15 North Third Ring East Road}, 
	city={Beijing},
	postcode={100029}, 
	state={Beijing},
	country={China}}

\author{Weili Guo\corref{cor1}}
\ead{guowl@mail.buct.edu.cn}

\cortext[cor1]{Corresponding author}
	
\author{Guangfeng Jiang}
\ead{jianggf@mail.buct.edu.cn }

\begin{abstract}
The critical group of a connected graph is closely related to the graph Laplacian, and is of high research value in combinatorics, algebraic geometry, statistical physics, and several other areas of mathematics. In this paper, we study the $k$-partite graphs and introduce an algorithm to get the structure of their critical groups by calculating the Smith normal forms of their graph Laplacians. When $k$ is from $2$ to $6$, we characterize the structure of the critical groups completely, which can generalize the results of the complete bipartite graphs. 
\end{abstract}


\end{frontmatter}

\section{Introduction}\label{Introduction}

Chip firing game is a discrete dynamical model	 studied by physicists, in the context of self-organized critical state. The basic rule of the model is that chips (sand, dollars) are exchanged between the sites in a network. When the model system reaches a particular state, even a very small perturbation can lead to collapse. For instance, the addition of a grain of sand can cause a massive avalanche.  In nature, there are a huge number of similar phenomena, such as fires, earthquakes, extinction of species and many others. The research of such phenomena is of vital significance, while a large number of scientists are interested and have achieved numerous results including a group structure, which is an important algebraic invariant associated with the chip-firing process. 

In 1990, Lorenzini \cite{lorenzini1991finite} named the group of components $\Phi(G) $ to approach the chip firing game from the viewpoint of arithmetic geometry. Dhar \cite{Dhar1990Self} defined the sandpile group from the perspective of physics. In 1997, Bacher \cite{bacher1997lattice} named this group Jacobian and Picard group, when working on the various lattices formed by graphs. In 1999, Biggs \cite{biggs1999chip} defined this group as critical group,  when doing research on the economic process  under the  theory  of chip firing.  

Furthermore, the critical group of a graph is strictly associated with the structure of the graph. From the Kirchoff's Matrix Tree Theorem \cite{biggs1993algebraic}, we get the following two  formulas.  

\begin{itemize}
	\item[(i)]  The order $\kappa(G)$  of the critical group of a graph $G$ is equal to the number of spanning trees in the graph,  $$\kappa(G)=(-1)^{i+j} \operatorname{det} \overline{L(G)}, $$
	
	where  $\overline{L(G)}$  is a reduced Laplacian matrix obtained from  $L(G)$  by striking out any row  $i$  and column  $j$.
	
	\item[(ii)] If the eigenvalues of  $L(G)$  are indexed  $\lambda_{1}, \ldots, \lambda_{n-1}, \lambda_{n}$, where  $n$ is the number of vertices of $G$ and  $\lambda_{n}=0$, then
	$$\kappa(G)=\frac{\lambda_{1} \cdots \lambda_{n-1}}{n} .$$
	
\end{itemize}

In (i), we note that the critical group can be used to study the corresponding graph. Besides, the critical group of a connected graph is a  finite Abelian group. In 1990, Rushanan \cite{rushanan1990combinatorial} found the comparable group related to the Smith normal form of adjacency matrices known as the Smith group. Then the algebraic structure of the critical group of a graph can be known from the Smith normal form of the Laplace matrix (or adjacency matrix). For a matrix, we can get the Smith normal form by the following row and column operations:

\begin{itemize}
	\item Add a non-zero integer multiple of one row (resp. column) to another row (resp. column), 
	
	\item Permute rows or colums, 
	
	\item Multiply a row or column by  $- 1$.
\end{itemize}

The critical group structures of some special  graphs are presently fully characterized, such as  the cycle graphs $C_{n}$ \cite{MERRIS1992181}, the complete graphs $K_{n}$ \cite{biggs1999chip}, the wheel graphs $W_{n}$ \cite{CORI2000447}, the bipartite graphs $K_{n_{1}, n_{2}}$ \cite{lorenzini1991finite},  the complete
multipartite graphs  $K_{n_{1}, \cdots , n_{k}}$  \cite{jacobson2003critical}, the de Bruijn graphs $DB(n,d)$ \cite{chan2015sandpile}, the M$\ddot{o}$bius ladders $M(n)$ \cite{deryagina2014jacobian}, the square cycles $C_{n}^{2}$ \cite{hou2006sandpile}, the threshold graphs \cite{christianson2002critical}, the $3 \times n$ twisted bracelets \cite{shen2008sandpile}, the n-cubes $Q_{n}$ \cite{bai2003critical}, the tree graphs \cite{levine2009sandpile}, the polygon flowers \cite{chen2019sandpile} and so on.
Moreover, there are also composite graphs such as the cartesian products of complete graphs \cite{jacobson2003critical}, $P_{4} \times C_{n}$ \cite{chen2008sandpile}, $K_{3} \times C_{n}$ \cite{2008On}, $K_{m} \times P_{n}$ \cite{liang2008critical},  $P_{m} \vee  P_{n}$ \cite{pan2011critical}  and so on. 

Base on the present researches, we study the  critical groups of a category of incomplete multipartite graphs which are introduced after the Definition \ref{d1.3}, and our work includes the results of the bipartite graphs \cite{lorenzini1991finite}. For the $k$-partite graph $G_{n_{1}, \ldots, {n_{k}}} $, we supply the algorithm of the critical groups. Furthermore, the specific abelian groups of $k$-partite graphs isomorphic to the critical groups are computed and listed, when $k=2, 3, \cdots , 6$.

This paper is organized as follows.  In the second section, we show the definitions and the invertible matrices associated with the row and column operations to get the simpler matrices $ L_{3},  L_{4} $, which can simplify the calculations to get the invariant factors of the critical groups $K(G_{n_{1}, \ldots, n_{k}})$. Through	the algorithms, we can achieve the structures of critical groups in the case $k=2, 3, \cdots , 6$ in the next sections. 

\section{The critical group of the $k$-partite graph $G_{n_{1}, \ldots, {n_{k}}} $}\label{sec2}
\bdefinition\label{d1.1}
Let $ G=(V, E)$  be a graph on  $n$  vertices. The \textbf{graph Laplacian } $L(G)$  is the  $n \times n$  matrix given by

\begin{equation*}
	L(G)_{i j}=\left\{\begin{array}{ll}
		-1, \quad \quad & i \neq j \text { and }\left\{v_{i}, v_{j}\right\} \in E; \\
		\operatorname{deg}\left(v_{i}\right), \quad \quad & i=j; \\
		0, \quad \quad & \text { otherwise. }
	\end{array}\right.
\end{equation*}

\edefinition

Let  $ A $  be the  $n \times n$  adjacency matrix of  $G$  and let $ D $ be the  $n \times n$  diagonal matrix with diagonal given by the degree sequence of  $G$. Then the above definition can be written as

$$L(G)=D-A . $$

When $ G$  is connected, the kernel of  $L(G)$  is spanned by the vectors in  $\mathbb{R}^{|V|}$  which are constant on the vertices. 

\bdefinition\label{d1.2}
Thinking of  $L(G)$  as a map  $\mathbb{Z}^{|V|} \rightarrow \mathbb{Z}^{|V|}$, its cokernel has the form

	$$\mathbb{Z}^{|V|} / \operatorname{im} L(G) \cong \mathbb{Z} \oplus K(G), $$
where  $K(G)$  is defined to be the \textbf{critical group}. 
\edefinition

For more details, please refer to  \cite{klivans2018mathematics}. 

\bdefinition\label{d1.3}
A \textbf{$k$-partite graph} is one whose vertex set can be partitioned into $k$ subsets, or parts, in such a way that no edge has both ends in the same part. 
\edefinition

In this article, we consider one kind of  $k$-partite graph $G$ with parts of sizes $n_{1} , n_{2}, $
 $ \cdots , n_{k} $. Meanwhile, $G$ is an incomplete graph, in which the vertices in the $i$-th subset are only adjacent to all vertices in the $(i-1)$-th and $(i+1)$-th subsets ($i=2, 3, \cdots ,  k-1$). Specifically, the vertices in the first subset are only adjacent to all vertices in the second subset. Similarly, the vertices in the $k$-th subset are only adjacent to all vertices in the $(k-1)$-th subset. For example, while $k=5, n_{1}=6 , n_{2}=4, n_{3}=5, n_{4}=3, n_{5}=4 $, $G_{n_{1}, n_{2}, \dots , n_{5}}$ is shown in the Figure \ref{fig:1}.

\begin{figure}[ht]
	\centering
	\includegraphics[scale=0.95]{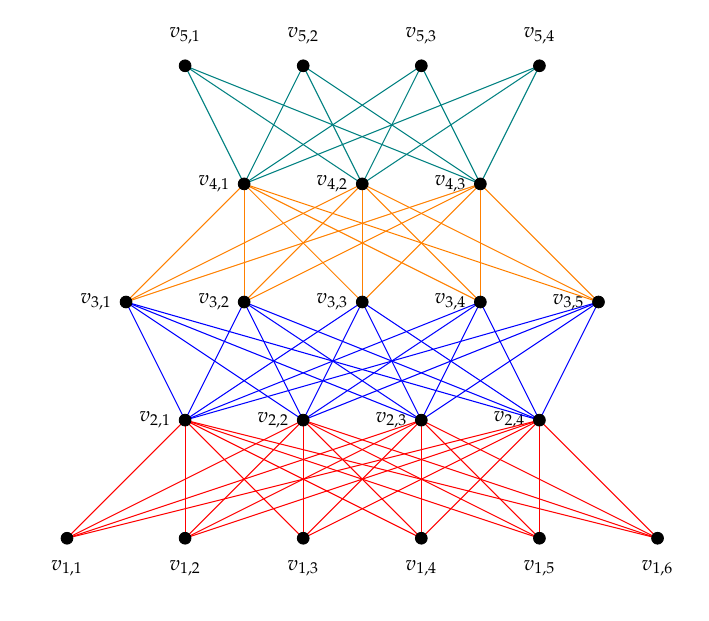}
	\caption{A $5$-partite graph ($n_{1}=6 , n_{2}=4, n_{3}=5, n_{4}=3, n_{5}=4 $). }
	\label{fig:1}
\end{figure}

For the sake of notation, let  $I_{n}$  denote an  $n \times n$  identity matrix, $O$ a zero matrix, and  $J_{m \times n} $ an  $m \times n$  matrix with all entries equal to $1$. Then it is easily seen that by ordering the vertices of  $G_{n_{1}, \ldots, n_{k}} $ in their groups of size  $n_{1}, n_{2}, \ldots, n_{k}$, one has

{\footnotesize  {
\begin{equation}
	L\left(G_{n_{1}, \ldots, {n_{k}}}\right)=\left[\begin{array}{cccccc}
		n_{2} I_{n_{1}} & -J_{n_{1} \times n_{2}} & O& \cdots & O&O \\
		-J_{n_{2} \times n_{1}} & (n_{1}+n_{3})I_{n_{2}} & -J_{n_{2} \times n_{3}}& \cdots & O&O \\
		O&-J_{n_{3} \times n_{2}} &(n_{2}+n_{4})I_{n_{3}} &\cdots & O&O \\
		\vdots &\vdots &\ddots &\ddots&\vdots &\vdots \\
		O&O&\cdots & -J_{n_{k-1} \times n_{k-2}} &(n_{k-2}+n_{k})I_{n_{k-1}}&-J_{n_{k-1} \times n_{k}} \\
		O&O&\cdots &O& -J_{n_{k} \times n_{k-1}} &n_{k-1}I_{n_{k}}		
	\end{array}\right].
	\label{Laplace0}
\end{equation}
}}

In the first stage of reduction, one can perform row and column operations on  $L\left(G_{n_{1}, \ldots, n_{k}}\right)$  to make 
\begin{equation}
	P_{1} 	L\left(G_{n_{1}, \ldots, {n_{k}}}\right)   Q_{1}=\left[\begin{array}{cccccc}
		L_{1,1} & L_{2,2} & O& \cdots & O&O \\
		L_{2,1}  & L_{1,2} & L_{2,3}& \cdots & O&O \\
		O&L_{2,2} &L_{1,3} &\cdots & O&O \\
		\vdots &\vdots &\ddots &\ddots&\vdots &\vdots \\
		O&O&\cdots &L_{2,k-2}  &L_{1,k-1}&L_{2,k} \\
		O&O&\cdots &O& L_{2,k-1} &L_{1,k}		
	\end{array}\right],
	\label{on-dig1}
\end{equation}
where
\begin{equation}
	L_{1,i}=	\left[\begin{array}{cccccc}
		N_{i} & 0 & 0 & \cdots & 0 & 0 \\
		0 & N_{i} & 0 & \cdots & 0 & 0 \\
		0 & 0 & N_{i} & \ddots & \vdots & \vdots \\
		\vdots & \vdots & \ddots & \ddots & 0 & 0 \\
		0 & 0 & \cdots & 0 & N_{i} & 0 \\
		0 & 0 & \cdots & 0 & 0 & N_{i}
	\end{array}\right],
	L_{2,i}=\left[\begin{array}{cccccc}
		-n_{i} & 0 & 0 & \cdots & 0 & -1 \\
		0 & 0 & 0 & \cdots & 0 & 0 \\
		0 & 0 & 0 & \ddots & \vdots & \vdots \\
		\vdots & \vdots & \ddots & \ddots & 0 & 0 \\
		0 & 0 & \cdots & 0 & 0 & 0 \\
		0 & 0 & \cdots & 0 & 0 & 0
	\end{array}\right].
\end{equation}
\begin{equation*}
	N_{i}=\left\{\begin{array}{ll}
		(n_{i-1}+n_{i+1}), & i=2,3,\cdots,k-1 ; \\
		n_{2}, & i=1; \\
		n_{k-1}, & i=k.
	\end{array}\right.
\end{equation*}


The matrices  $P_{1}$  and  $Q_{1}$  are block diagonal  $P_{1}=\operatorname{diag}\left(P_{1,1}, \ldots, P_{1,k}\right)$,\\
 $Q_{1}=\operatorname{diag}\left(Q_{1,1}, \ldots, Q_{1,k}\right)$, where  $P_{1,i}$  and  $Q_{1,i}$  are  $n_{i} \times n_{i}$  matrices given as:

\begin{equation}
	P_{1,i}=\left[\begin{array}{cccccc}
		1 & 0 & 0 & \cdots & 0 & 0 \\
		-1 & 1 & 0 & \cdots & 0 & 0 \\
		0 & -1 & 1 & \ddots & \vdots & \vdots \\
		\vdots & \vdots & \ddots & \ddots & 0 & 0 \\
		0 & 0 & \cdots & -1 & 1 & 0 \\
		-n_{i}+1 & 1 & \cdots & 1 & 1 & 1
	\end{array}\right], \quad  Q_{1,i}=\left[\begin{array}{cccccc}
		1 & 0 & 0 & \cdots & 0 & 0 \\
		1 & 1 & 0 & \cdots & 0 & 0 \\
		1 & 1 & 1 & \cdots & \vdots & \vdots \\
		\vdots & \vdots & \ddots & \ddots & 0 & 0 \\
		1 & 1 & \cdots & 1 & 1 & 0 \\
		1 & -n_{i}+2 & \cdots & -2 & -1 & 1
	\end{array}\right]. 
\end{equation}

According to the above row and column operations on $L(G_{n_{1}, \ldots, n_{k}})$, we can get the following proposition. 

\begin{proposition}
	The critical group of the graph $G_{n_{1}, \ldots, n_{k}}$ has the following isomorphism, 
	\begin{equation}
		\mathbb{Z} \oplus K(G_{n_{1}, \ldots, n_{k}})\cong  \left(\bigoplus_{i=1}^{k} \mathbb{Z} /  (N_{i} \mathbb{Z})^{\oplus(n_{i}-2)} \right) \oplus coker L_{3}, 
	\end{equation}
	where  $L_{3}$  is the  $2k \times 2k$  matrix obtained by removing some rows and columns 
	
	\begin{equation}
		L_{3}=\left[\begin{array}{cccccccc}
			N_{1} & 0 & -n_{2} & -1 & \ldots & \ldots & 0 & 0 \\
			0 & N_{1} & 0 & 0 & \cdots  & \ldots& 0 & 0 \\
			-n_{1} & -1 & N_{2} & 0 & \ddots & \ddots& \vdots& \vdots \\
			0 & 0 & 0 & N_{2} & \ddots&\ddots &\vdots &\vdots \\
			\vdots& \vdots &\ddots & \ddots &\ddots & \ddots & -n_{k} & -1 \\
			\vdots& \vdots & & \ddots &\ddots & \ddots & 0 & 0 \\
			0 & 0 &\cdots &\cdots & -n_{k-1} & -1 & N_{k} & 0 \\
			0 & 0 &\cdots & \cdots & 0 & 0 & 0 & N_{k}
		\end{array}\right].
		\label{L3}
	\end{equation}
\end{proposition}

		\begin{proof}
			After the operation $P_{1} 	L\left(G_{n_{1}, \ldots, {n_{k}}}\right)   Q_{1}$, the resulting matrix is as follows: 
			\begin{equation}
				\overline{L} =\left[\begin{array}{cccccccccccccc}
					N_{1} &0 &\ldots & 0 & -n_{2} & 0& \ldots&-1 & \ldots & \ldots & 0 &0&\ldots& 0 \\
					0 & N_{1} &\ldots& 0 & 0 &0&\cdots&0& \cdots  & \ldots& 0 & 0 &\ldots& 0\\
					\vdots& \vdots&\ddots&\vdots&\vdots& \vdots&\ddots&\vdots&\cdots& \cdots&\vdots&\vdots&\ddots&\vdots\\
					0 & 0 &\ldots& N_{1} & 0 &0&\cdots&0& \cdots  & \ldots& 0 & 0 &\ldots& 0\\
					-n_{1} &0&\cdots& -1 & N_{2} & 0 & \cdots&0 & \cdots& \cdots&0&0& \cdots &0\\
					0 & 0 &\cdots& 0&0 & N_{2} & \cdots&0&\cdots &\cdots &0&0&\cdots &0\\
					\vdots& \vdots&\ddots&\vdots&\vdots& \vdots&\ddots&\vdots&\cdots& \cdots&\vdots&\vdots&\ddots&\vdots\\
					0 & 0 &\ldots& 0 & 0 &0&\cdots&N_{2}& \cdots  & \ldots& 0 & 0 &\ldots& 0\\
					0 &0 &\ldots & 0 & -n_{2} & 0& \ldots&-1 & \ldots & \ldots & 0 &0&\ldots& 0 \\
					0 & 0 &\ldots& 0 & 0 &0&\cdots&0& \cdots  & \ldots& 0 & 0 &\ldots& 0\\
					\vdots& \vdots&\ddots&\vdots&\vdots& \vdots&\ddots&\vdots&\cdots& \cdots&\vdots&\vdots&\ddots&\vdots\\
					0 & 0 &\ldots& 0 & 0 &0&\cdots&0& \cdots  & \ldots& 0 & 0 &\ldots& 0\\
					\vdots& \vdots& &\vdots&\vdots&\vdots&&\vdots&&&\vdots&\vdots&&\vdots\\
					\vdots& \vdots& &\vdots&\vdots&\vdots&&\vdots&&&\vdots&\vdots&&\vdots\\			
					0 &0 &\ldots & 0 & 0 & 0& \ldots&0 & \ldots & \ldots & N_{k} &0&\ldots& 0 \\
					0 & 0 &\ldots& 0 & 0 &0&\cdots&0& \cdots  & \ldots& 0 & N_{k} &\ldots& 0\\
					\vdots& \vdots&\ddots&\vdots&\vdots& \vdots&\ddots&\vdots&\cdots& \cdots&\vdots&\vdots&\ddots&\vdots\\
					0 & 0 &\ldots& 0 & 0 &0&\cdots&0& \cdots  & \ldots& 0 & 0 &\ldots& N_{k}\\
				\end{array}\right].
				\label{L_bar}
			\end{equation}
			
			Consider the rows and columns of the integers $N_{1}$ in the matrix, we can find that the entries are only zeros in the same rows and columns as from the second to $(n_{1}-1)$-th entry $N_{1}$. The situations are same for the integer from $N_{2}$ to $N_{k}$. Hence, we can  obtain $n_{i}-2$ invariant factors $N_{i}$ and the  $2k \times 2k$  matrix $L_{3}$ by removing these rows and columns.
			$\hfill\blacksquare$ 
			
		\end{proof}

By calculation, we can obtain $L_{4}=P_{2}L_{3}Q_{2}$, where $P_{2}, Q_{2} \in G L_{2 k}(\mathbb{Z})$   are as follows,  

\begin{equation}
	\begin{array}{l} 
		P_{2}=\left[\begin{array}{cccccccc}
			B & A & O & O & O &  \cdots & \cdots&O \\
			n_{1}B-D & n_{2}B & A & O & O &  \cdots & \cdots&O \\
			n_{1}B-D &	n_{2}B-D & n_{3}B & A & O   & \cdots & \cdots&O \\
			\vdots & \vdots & \ddots & \ddots & \ddots &\ddots & & \vdots  \\
			\vdots & \vdots & & \ddots & \ddots & \ddots &\ddots &  \vdots  \\
			\vdots & \vdots & & & \ddots & \ddots & \ddots & O \\
			n_{1}B-D & n_{2}B-D & \cdots & \cdots & \cdots  & n_{k-2}B-D &n_{k-1}B& A \\
			n_{1}R+S & n_{2}R+S & \cdots & \cdots & \cdots  & n_{k-2}R+S & n_{k-1}R+S &n_{k}R+T
		\end{array}\right],
	\end{array}
	\label{P2}
\end{equation}

	\begin{equation}
		\begin{array}{l} 
			Q_{2}=	\left[\begin{array}{cccccc}
				-n_{1}A+I_{2} & O & O & \cdots & O & O \\
				O & -n_{2}A+I_{2} & O & \cdots & O & O \\
				O & O & -n_{3}A+I_{2} & \ddots & \vdots & \vdots \\
				\vdots & \vdots & \ddots & \ddots & O & O \\
				O & O & \cdots & O & -n_{k-1}A+I_{2} & O \\
				O & O & \cdots & O & O & -n_{k}A+I_{2}
			\end{array}\right], 
		\end{array}
		\label{Q2}
	\end{equation}
where

$$ A= \left[\begin{array}{cc}
	0 & 0 \\
	1 & 0
\end{array}\right], \quad B= \left[\begin{array}{cc}
	1 & 0 \\
	0 & 0
\end{array}\right], \quad C= \left[\begin{array}{cc}
	1 & 0 \\
	0 & -1
\end{array}\right],\quad  D= \left[\begin{array}{cc}
	0 & -1 \\
	0 & 0
\end{array}\right], $$
~\\
$$ R= \left[\begin{array}{cc}
	1 & 0 \\
	1 & 0
\end{array}\right], \quad S= \left[\begin{array}{cc}
	0 & 1 \\
	0 & 1
\end{array}\right], \quad T= \left[\begin{array}{cc}
	0 & 0 \\
	0 & 1
\end{array}\right]. $$

Further reduction of  $L_{3}$  can be achieved by re-ordering rows and columns to obtain
{\tiny
	\begin{equation}
		L_{4}=\left[\begin{array}{cccccccccccccccc}
			-n_{2}B-T &  N_{2}A+D &  C & O & O & O    & \cdots  & \cdots & O  \\
			O & n_{2}N_{2}B+n_{1}D-T & N_{3}A+n_{2}D & C & O & O & \cdots  & \cdots & O\\
			
			O& O & n_{3}N_{3}B+n_{2}D-T & N_{4}A+n_{3}D &  C & O  &\cdots  & \cdots & O \\
			
			\vdots & \vdots & \ddots & \ddots & \ddots  & \ddots&  &  \vdots &\vdots\\
			
			\vdots & \vdots & & \ddots  & \ddots & \ddots  & \ddots &   \vdots &\vdots\\
			
			\vdots & \vdots & & &\ddots  & \ddots & \ddots  & \ddots &   \vdots \\
			
			O&O& \cdots&\cdots&   \cdots&O  &  n_{k-2}N_{k-2}B+n_{k-3}D-T &N_{k-1}A+n_{k-2}D&C\\
			
			O&O& \cdots&\cdots&  \cdots& \cdots&O  & n_{k-1}N_{k-1}B+n_{k-2}D-T&N_{k}A+n_{k-1}D\\
			
			O&O& \cdots&\cdots&  \cdots& \cdots&O& O & n_{k}N_{k}B+n_{k-1}D \\
		\end{array}\right] .
		\label{L4}
	\end{equation}	
}

Now, $ L_{4} $ is an upper triangular matrix and  upper $5$-banded matrix, where the entries in the $i$-th row and $j$-th column are zeros for $j<i$ and $j\ge i+5 $. According to the algorithms in the paper \cite{jager2004new}, we can reduce $ L_{4} $ to an upper $2$-banded matrix. And then we obtain its Smith normal form  by the algorithms in \cite{storjohann1996near}. 

\begin{proposition}
	By the above steps, for $k\ge 4$, the critical groups can be decomposed as 
	
	\begin{equation}
		\begin{aligned}
			\mathbb{Z} \oplus K(G_{n_{1}, \ldots, n_{k}})\cong & \left( \bigoplus_{i=1}^{k}  \mathbb{Z} /  (N_{i} \mathbb{Z})^{\oplus(n_{i}-2)}\right)   \oplus \mathbb{Z} /  (n_{2}(n_{1}+n_{3}) \mathbb{Z}) \\
			& \oplus \mathbb{Z} /  (n_{k-2}(n_{k-1}+n_{k}) \mathbb{Z}) 
			\oplus \mathcal{G},\\ 
		\end{aligned}
	\end{equation}
	where $\mathcal{G} $ is a finite  Abelian group determined by the numbers $n_{1}, n_{2}, \cdots,  n_{k}$. And we can achieve the determinant of the graph Laplacian $L(G_{n_{1}, \ldots, {n_{k}}})$  is 
	\begin{equation}
		det(L(G_{n_{1}, \ldots, {n_{k}}}))= \left(\prod_{i=1}^{k} N_{i}^{n_{i}-1}  \right) \cdot \left(\prod_{i=2}^{k-1} n_{i} \right) 
	\end{equation} 
	which is the number of spanning trees in the graph $G_{n_{1}, \ldots, {n_{k}}}$. 
\end{proposition}

\begin{example}
	For the $2$-partite graph  $ G_{n_{1},n_{2}} $, we can get the followings from  Equation~(\ref{Laplace0}) to  Equation~(\ref{L3}),
	\begin{equation}
		\mathbb{Z} \oplus K(G_{n_{1}, n_{2}})\cong  \mathbb{Z} /  (n_{1} \mathbb{Z})^{\oplus(n_{2}-2)} \oplus \mathbb{Z} /  (n_{2} \mathbb{Z})^{\oplus(n_{1}-2)} \oplus coker L_{3}(G_{n_{1}, n_{2}}),
	\end{equation}
	\begin{equation}
		L_{3}(G_{n_{1},n_{2}}) =\left[\begin{array}{cccccccc}
			n_{2} & 0 & -n_{2} & -1 \\
			0 & n_{2} & 0 & 0 \\
			-n_{1} & -1 & n_{1} & 0 \\
			0 & 0 & 0 & n_{1} 			
		\end{array}\right].
	\end{equation}
	By some row and column operations, we get $L_{4}(G_{n_{1},n_{2}}) = P_{3}(G_{n_{1},n_{2}})L_{3}(G_{n_{1},n_{2}})Q_{3}(G_{n_{1},n_{2}}) $, 
	where 
		\begin{equation*}
	\begin{array}{l} 
		P_{3}(G_{n_{1}, n_{2}})=\left[\begin{array}{cccc}
			1     & 0 & 0 & 0 \\
			0     &1  & 0 & 0 \\
			-n_{1} &0  & 1 & 0 \\
			0     & 0 & 0 & 1 
		\end{array}\right]
	\end{array}, 
	\begin{array}{l} 
		Q_{3}(G_{n_{1},  n_{2}}) =\left[\begin{array}{cccc}
			0     & 0 & 0 & 1 \\
			0 &1 & n_{1}    & n_{1}  \\
			0     & 0 & 1 & 1 \\
			1 & 0 & 0 & n_{2} \\
		\end{array}\right], 
	\end{array}
\end{equation*}
\begin{equation}
	\begin{array}{l} 
		L_{4}(G_{n_{1}, n_{2}})=\left[\begin{array}{cccc}
			-1     & 0 & 0 & 0 \\
			0 &-1  & 0 & 0 \\
			0 &0  & n_{1}n_{2} & 0 \\
			0     & 0 & 0 & 0
		\end{array}\right]
	\end{array}. 
\end{equation}
	
	Then we find
	\begin{equation}
		\mathbb{Z} \oplus K(G_{n_{1}, n_{2}})\cong  \mathbb{Z} /  (n_{1} \mathbb{Z})^{\oplus(n_{2}-2)} \oplus \mathbb{Z} /  (n_{2} \mathbb{Z})^{\oplus(n_{1}-2)} \oplus \mathbb{Z} /  (n_{1}n_{2} \mathbb{Z}).
	\end{equation}

\end{example}

In this case, the result is identical to the one in \cite{lorenzini1991finite}. It is straightforward to work out the critical group structures of complete bipartite graphs with our method.

\begin{remark}
	For $k=3$, $ G_{n_{1},n_{2},n_{3}} $ is also a  complete bipartite graph. In other words, consider the $n_{1}+n_{3}$ vertices  of the first and third parts as one part of the bipartite graph, and the remaining $n_{2}$ vertices as the other part.
\end{remark}

\section{The critical group of the $4$-partite graph}\label{Critical group of 4-partite graph}

In this section, we can obtain the critical group of the $4$-partite graph $G_{n_{1}, \ldots, n_{4}}$. Following the above calculation steps from Equation~(\ref{Laplace0}) to Equation~(\ref{L4}), we can achieve

\begin{equation}
	\mathbb{Z} \oplus K(G_{n_{1}, \ldots, n_{4}})\cong \left(\bigoplus_{i=1}^{4} \mathbb{Z} /  (N_{i} \mathbb{Z})^{\oplus(n_{i}-2)} \right)  \oplus coker L_{4}(G_{n_{1}, \ldots, n_{4}}),
\end{equation}
where 	\begin{equation}
	L_{4}(G_{n_{1}, \ldots, n_{4}}):=\left[\begin{array}{cccccccc}
		n_{2} & 0 & 0 & -1 & 0 & 0 & 0 & 0\\
		0 & -1 & n_{1}+n_{3} & 0 & 0 & -1 & 0 & 0\\ 
		0 & 0 & n_{2}(n_{1}+n_{3}) & -n_{1} & 0 & n_{2} & 0 & 0\\
		0 & 0 & 0 & -1 & n_{2}+n_{4} & 0 & 0 & -1\\
		0 & 0 & 0 & 0 & n_{3}(n_{2}+n_{4}) & -n_{2} & 0 & -n_{3}\\
		0 & 0 & 0 & 0 & 0 & -1 & n_{3} & 0\\
		0 & 0 & 0 & 0 & 0 &  0 & n_{3}n_{4} & -n_{3}\\
		0 & 0 & 0 & 0 & 0 &  0 & 0 & 0\\
	\end{array}\right] .
	\label{L''4}
\end{equation}

By calculation, we can obtain $L_{5}(G_{n_{1}, \ldots, n_{4}})=P_{3}(G_{n_{1}, \ldots, n_{4}})L_{4}(G_{n_{1}, \ldots, n_{4}})Q_{3}(G_{n_{1}, \ldots, n_{4}})$, where $P_{3}(G_{n_{1}, \ldots, n_{4}}), Q_{3}(G_{n_{1}, \ldots, n_{4}}) \in G L_{8}(\mathbb{Z})$  are

\begin{equation}
	\begin{array}{l} 
		P_{3}(G_{n_{1}, \ldots, n_{4}})=\left[\begin{array}{cccccccc}
			1 & 0 & 0 & -1 & 0 & 0 & 0 & 0 \\
			0 & 1 & 0 & 0  & 0 & 0 & 0 & 0 \\
			-n_{1}-n_{3} & 0 & 1 & n_{3} & -1 & 0 & 0 & 0 \\
			0 & 0 & 0 & 1 & 0 & 0 & 0 & 0 \\
			0 & 0 & 0 & 0 & 1 & -n_{2} & -1 & 0 \\
			0 & 0 & 0 & 0 & 0 & 1 & 0 & 0 \\
			n_{3} & 0 & 0 & -n_{3} & 1 & -n_{2} & 0 & 0 \\
			0 & 0 & 0 & 0 & 0 & 0 & 0 & 1 \\
			
		\end{array}\right]
	\end{array}, 
\end{equation}
\begin{equation}
	\begin{array}{l} 
		Q_{3}(G_{n_{1}, \ldots, n_{4}}) =\left[\begin{array}{cccccccc}
			0 & 0 & 1 & 0 & 0 & 0 & 0 & 1 \\
			0 & 1 & -n_{3} & 0 & 0 & -1 & -n_{3} & n_{1} \\
			0 & 0 & 0 & 0 & 0 & 0 & 0 & 1 \\
			-1 & 0 & n_{2} & 1 & 0 & 0 & 0 & n_{2} \\
			0 & 0 & 1 & 0 & 1 & 0 & 1 & 1 \\
			0 & 0 & n_{3} & 0 & 0 & 1 & n_{3} & n_{3} \\
			0 & 0 & 1 & 0 & 0 & 0 & 1 & 1 \\
			1 & 0 & n_{4} & 0 & n_{2}+n_{4} & 0 & n_{2}+n_{4} & n_{4} \\				
		\end{array}\right], 
	\end{array}
\end{equation}

and

\begin{equation}
	L_{5}(G_{n_{1}, \ldots, n_{4}}):=\left[\begin{array}{cccccccc}
		1 & 0 & 0 & 0 & 0 & 0 & 0 &  0\\
		0 & -1 & 0 & 0 & 0 & 0 & 0 & 0 \\
		0 & 0 & n_{2}(n_{1}+n_{3}) & 0 & 0 & 0 & 0 & 0 \\
		0 & 0 & 0 & -1 & 0 & 0 & 0 & 0 \\
		0 & 0 & 0 & 0 & n_{3}(n_{2}+n_{4}) & 0 & 0 & 0 \\
		0 & 0 & 0 & 0 & 0 & -1 & 0 & 0 \\
		0 & 0 & 0 & 0 & 0 & 0 & -n_{2}n_{3} & 0 \\
		0 & 0 & 0 & 0 & 0 & 0 & 0 & 0 \\
	\end{array}\right] .
\end{equation}	

Hence, we can obtain the following therom. 

\begin{theorem}
	The critical group of $G_{n_{1},n_{2},n_{3},n_{4}}$ has the following structure
	\begin{equation}
		\begin{split}
			\mathbb{Z} \oplus K(G_{n_{1},n_{2},n_{3},n_{4}})\cong &
			\mathbb{Z} /  (n_{2} \mathbb{Z})^{\oplus(n_{1}-2)} \oplus \mathbb{Z} /  ((n_{1}+n_{3}) \mathbb{Z})^{\oplus(n_{2}-2)}\\
			 & \oplus \mathbb{Z} /  ((n_{2}+n_{4}) \mathbb{Z})^{\oplus(n_{3}-2)}
			 \oplus \mathbb{Z} /  (n_{3} \mathbb{Z})^{\oplus(n_{4}-2)} \\
			 &\oplus \mathbb{Z} /  (n_{2}n_{3}) \mathbb{Z}) \oplus \mathbb{Z} /  (n_{2}(n_{1}+n_{3}) \mathbb{Z}) \oplus \mathbb{Z} /  (n_{3}(n_{2}+n_{4}) \mathbb{Z}).
		\end{split}
	\end{equation}
\end{theorem}


\section{The critical group of  the $5$-partite graph}\label{Critical group of 5-partite graph}

In this section, we continue to calculate the critical group of the $5$-partite graph with the same method as before. Then we can get 

\begin{equation}
	\mathbb{Z} \oplus K(G_{n_{1}, \ldots, n_{5}})\cong  \left(\bigoplus_{i=1}^{5} \mathbb{Z} /  (N_{i} \mathbb{Z})^{\oplus(n_{i}-2)} \right) \oplus coker L_{4}(G_{n_{1}, \ldots, n_{5}}),
\end{equation}
where 	{\footnotesize 
\begin{equation}
	L_{4}(G_{n_{1}, \ldots, n_{5}}):=\left[\begin{array}{cccccccccc}
		n_{2} & 0&0         &-1 &0 &0 &0&0&0&0\\
		0     &-1&n_{1}+n_{3}&0  &0 &-1&0&0&0&0\\
		0     &0 &n_{2}(n_{1}+n_{3})&-n_{1}&0&-n_{2}&0&0&0&0\\
		0&0&0&-1&n_{2}+n_{4}&0&0&-1&0&0\\
		0&0&0&0&n_{3}(n_{2}+n_{4})&-n_{2}&0&-n_{3}&0&0\\
		0&0&0&0&0&-1&n_{4}&0&0&-1\\
		0&0&0&0&0&0&n_{4}(n_{3}+n_{5})&-n_{3}&0&-n_{4}\\
		0&0&0&0&0&0&0&-1&n_{4}&0\\
		0&0&0&0&0&0&0&0&n_{4}n_{5}&-n_{4}\\
		0&0&0&0&0&0&0&0&0&0\\
	\end{array}\right] .
	\label{L''4}
\end{equation}
}

By calculation ,we can obtain $L_{5}(G_{n_{1}, \ldots, n_{5}})=P_{3}(G_{n_{1}, \ldots, n_{5}})L_{4}(G_{n_{1}, \ldots, n_{5}})Q_{3}(G_{n_{1}, \ldots, n_{5}})$, where $P_{3}(G_{n_{1}, \ldots, n_{5}}), Q_{3}(G_{n_{1}, \ldots, n_{5}}) \in G L_{10}(\mathbb{Z})$  are

\begin{equation}
	\begin{array}{l} 
		P_{3}(G_{n_{1}, \ldots, n_{5}})=\left[\begin{array}{cccccccccc}
			1     & 0 & 0 & -1 &0 & 0 & 0 & 1 & 0 &0\\
			0     & 1 & 0 & 0 &0 & 0 & 0 & 1 & 0 &0\\
			-n_{1}-n_{3} & 0 &1  & n_{3} & -1 & 0& 0&0&0&0\\
			0     & 0 & 0 & 1 &0 & 0 & 0 & 1 & 0 &0\\
			n_{3} & 0 &0  & -n_{3} & 1 & -n_{2}& 0&0&0&0\\
			0     & 0 & 0 & 0 &0 & 1 & 0 & 0 & 0 &0\\
			0     & 0 & 0 & 0 &0 & 0 & 1 & -n_{3} & -1 &0\\
			0     & 0 & 0 & 0 &0 & 0 & 0 & 1 & 0 &0\\
			0     & 0 & 0 & 0 &0 & 0 & 1 & -n_{3} & 0 &0\\
			0     & 0 & 0 & 0 &0 & 0 & 0 & 0 & 0 &1\\
		\end{array}\right]
	\end{array}, 
\end{equation}
\begin{equation}
	\begin{array}{l} 
		Q_{3}(G_{n_{1}, \ldots, n_{5}})=\left[\begin{array}{cccccccccc}
			0     & 0 & 0 & 0 & 1 & 0 &0 &0 &0&1\\
			0     & 1 & n_{1}+n_{3} & 0 & n_{1}+n_{3} & -1 &0 &0 &1&n_{1}\\
			0     & 0 & 1 & 0 & 1 & 0 &0 &0 &0&1\\
			n_{2}+n_{4}     & 0 & 0 & 1 & 0 & 0 &0 &-1 &0&n_{2}\\
			1     & 0 & 0 & 0 & 0 & 0 &0 &0 &0&1\\
			0     & 0 & 0 & 0 & 0 & 1 &0 &0 &-1&n_{3}\\
			0     & 0 & 0 & 0 & 0 & 0 &1 &0 &0&1\\
			0     & 0 & 0 & 0 & 0 & 0 &0 &1&0&n_{4}\\
			0     & 0 & 0 & 0 & 0 & 0 &0 &0 &0&1\\
			0     & 0 & 0 & 0 & 0 & 0 &n_{4} &0 &1&n_{5}\\
		\end{array}\right],
	\end{array}
\end{equation}

and
{\small 
\begin{equation}
	L_{5}(G_{n_{1}, \ldots, n_{5}})=\left[\begin{array}{cccccccccc}
		-n_{2}-n_{4} & 0 &0 &0 &n_{2} & 0 & 0 & 0 &0 &0\\
		0 & -1 & 0 & 0 & 0 & 0 & 0 & 0 & 0 & 0 \\
		0 & 0 & n_{2}(n_{1}+n_{3}) & 0 & 0 & 0 & 0 & 0 & 0 & 0 \\
		0 & 0 & 0 & -1 & 0 & 0 & 0 & 0 & 0 & 0 \\
		0 & 0 & 0 & 0 & n_{2}n_{3} & 0 & 0 & 0 &  n_{2} &0\\
		0 & 0 & 0 & 0 & 0 & -1 & 0 & 0 & 0 & 0 \\
		0 & 0 & 0 & 0 & 0 & 0 & n_{4}(n_{3}+n_{5}) & 0 & 0 & 0 \\
		0 & 0 & 0 & 0 & 0 & 0 & 0 & -1 & 0 & 0 \\
		0 & 0 & 0 & 0 & 0 & 0 & 0 & 0 &  -n_{4} &0\\
		0 & 0 & 0 & 0 & 0 & 0 & 0 & 0 & 0 & 0 
	\end{array}\right] .
\end{equation}	
}
	


By the row and column operations, we can further reduce  $	L_{5}(G_{n_{1}, \ldots, n_{5}})$ to obtain $ \left[\begin{array}{cc}
	L_{6} & O \\
	O & L_{7}
\end{array}\right]$, where $L_{6}$ is a diagonal matrix, and $ L_{7}=\left[\begin{array}{ccc}
-n_{2}-n_{4} & n_{2} & 0 \\
0 & n_{2}n_{4} &n_{2}\\
0 &0 &-n_{4}
\end{array}\right]$. By calculating the Smith normal form of $ L_{7}$, we get the invariant factors $ \sigma_{1}, \sigma_{2}/\sigma_{1}, $ $ det(L_{7})/\sigma_{2} $. Then we obtain the following theorem.

\begin{theorem}
	For the graph $ G_{n_{1},n_{2},n_{3},n_{4},n_{5}}$, its critical group can be decomposed as following 
	\begin{equation}
		\begin{split}
			\mathbb{Z} \oplus K(G_{n_{1},n_{2},n_{3},n_{4},n_{5}})\cong &
			\mathbb{Z} /  (n_{2} \mathbb{Z})^{\oplus(n_{1}-2)} \oplus \mathbb{Z} /  ((n_{1}+n_{3}) \mathbb{Z})^{\oplus(n_{2}-2)} \\
			& \oplus \mathbb{Z} /  ((n_{2}+n_{4}) \mathbb{Z})^{\oplus(n_{3}-2)} \oplus \mathbb{Z} /  ((n_{3}+n_{5}) \mathbb{Z})^{\oplus(n_{4}-2)} \\
			&\oplus \mathbb{Z} /  (n_{4} \mathbb{Z})^{\oplus(n_{5}-2)}  \oplus \mathbb{Z} /  (n_{2}(n_{1}+n_{3}) \mathbb{Z})\oplus \mathbb{Z} /  (n_{4}(n_{3}+n_{5})) \mathbb{Z}) \\&\oplus \mathbb{Z} /  (\sigma_{1} \mathbb{Z})\oplus \mathbb{Z} /  ((\sigma_{2}/\sigma_{1}) \mathbb{Z}) \oplus 	\mathbb{Z} /  ((n_{2}n_{3}n_{4}(n_{2}+n_{4}) /\sigma_{2}) \mathbb{Z}),
		\end{split}
	\end{equation} 
	where	\begin{align*}
		&\sigma_{1}=gcd(n_{2},n_{4},n_{2}+n_{4},n_{2}n_{3}), \\ 
		& \sigma_{2}=gcd(n_{2}^{2} , n_{2}n_{4} , n_{2}n_{3}n_{4} , n_{2}(n_{2}+n_{4}) , n_{4}(n_{2}+n_{4}) , n_{2}n_{3}(n_{2}+n_{4}) ).
	\end{align*} 	
\end{theorem}


\section{Discussion}\label{Discussion}
With the above method, for $k=6$, 
\begin{equation}
	\begin{split}
		\mathbb{Z} \oplus K(G_{n_{1},\cdots ,n_{6}})\cong &
		\mathbb{Z} /  (n_{2} \mathbb{Z})^{\oplus(n_{1}-2)} \oplus \mathbb{Z} /  ((n_{1}+n_{3}) \mathbb{Z})^{\oplus(n_{2}-2)} \\
		& \oplus \mathbb{Z} /  ((n_{2}+n_{4}) \mathbb{Z})^{\oplus(n_{3}-2)} \oplus \mathbb{Z} /  ((n_{3}+n_{5}) \mathbb{Z})^{\oplus(n_{4}-2)} \\
		&  \oplus \mathbb{Z} /  ((n_{4}+n_{6}) \mathbb{Z})^{\oplus(n_{5}-2)}  \oplus \mathbb{Z} /  (n_{5} \mathbb{Z})^{\oplus(n_{6}-2)} \\
		& \oplus \mathbb{Z} /  (n_{2}(n_{1}+n_{3}) \mathbb{Z}) \oplus \mathbb{Z} /  (n_{5}(n_{4}+n_{6})) \mathbb{Z}) \oplus \mathbb{Z} /  (\sigma_{1} \mathbb{Z}) \\&\oplus \mathbb{Z} /  ((\sigma_{2}/\sigma_{1}) \mathbb{Z})\oplus 	\mathbb{Z} /  ((n_{2}n_{3}n_{4}n_{5}(n_{2}+n_{4})(n_{3}+n_{5}) /\sigma_{2}) \mathbb{Z}).
	\end{split}
\end{equation} 
where \begin{equation*}
	\begin{split}
		\sigma_{1}=&gcd(n_{2}n_{3}, n_{2}n_{5}, n_{3}(n_{2}+n_{4}), n_{5}(n_{2}+n_{4}), n_{2}(n_{3}+n_{5}), n_{4}(n_{3}+n_{5})),\\
		\sigma_{2}=& gcd(n_{2}n_{3}^{2}(n_{2}+n_{4}) , n_{2}n_{3}n_{5}(n_{2}+n_{4}), n_{2}n_{3}(n_{2}+n_{4})(n_{3}+n_{5}), n_{2}^{2}n_{5}(n_{3}+n_{5}),  \\
		& n_{5}(n_{2}+n_{4})^{2}(n_{3}+n_{5}) ) .
	\end{split}
\end{equation*}

In this paper, we study the  critical group of the $k$-partite graph $G_{n_{1}, \ldots, n_{k}}$. First of all, we obtain the algorithm of the critical group $K(G_{n_{1}, \ldots, n_{k}})$ for the  arbitrary $k$.  When $k = 2$, $G_{n_{1}, n_{2}}$  is a completely bipartite graph, and our conclusion is consistent with the result in \cite{lorenzini1991finite}. Then the decompositions of the critical groups of $k$-partite graphs are given for the cases $k = 3, 4, 5, $ and $ 6$.

For further research, we have two questions. 

Question I: Based on the $k$-partite graphs in this paper, randomly deleting some edges, how to calculate the critical groups of the modified graphs? 

Question II: What is the solution to compute the critical groups for the arbitrary incomplete multi-partite graphs? 
\section{Acknowledgments} 
	We would like to show our great gratitude to the anonymous referees for carefully reading this manuscript and improving its presentation and accuracy. The corresponding author is supported by National Science Fund for Distinguished Young Scholars 12201029. 

\bibliographystyle{elsarticle-num} 
\bibliography{mybibfile}

\end{document}